\documentclass[12pt]{amsart}
\usepackage[utf8]{inputenc}
\usepackage{amssymb}
\usepackage{hyperref}
\usepackage[final]{showkeys}

\input xy
\xyoption{all}

\theoremstyle{definition}
\newtheorem{mydef}{Definition}[section]
\newtheorem{lem}[mydef]{Lemma}
\newtheorem{thm}[mydef]{Theorem}
\newtheorem{cor}[mydef]{Corollary}

\newtheorem{hypothesis}[mydef]{Hypothesis}
\newtheorem{prop}[mydef]{Proposition}
\newtheorem{defin}[mydef]{Definition}
\newtheorem{example}[mydef]{Example}
\newtheorem{remark}[mydef]{Remark}
\newtheorem{notation}[mydef]{Notation}
\newtheorem{fact}[mydef]{Fact}

\newcommand{\fct}[2]{{}^{#1}#2}

\newcommand{\ba}{\bar{a}}
\newcommand{\bb}{\bar{b}}
\newcommand{\bc}{\bar{c}}

\newcommand{\bx}{\bar{x}}
\newcommand{\by}{\bar{y}}

\newcommand{\bigM}{\widehat{M}}
\newcommand{\bigN}{\widehat{N}}

\newcommand{\bigL}{\widehat{L}}
\newcommand{\bigtau}{\widehat{\tau}}

\newcommand{\bigK}{\widehat{K}}

\newcommand{\sea}{\mathfrak{C}}
\newcommand{\dom}[1]{\text{dom}(#1)}

\newcommand{\cf}[1]{\text{cf} (#1)}
\newcommand{\seq}[1]{\langle #1 \rangle}
\newcommand{\rest}{\upharpoonright}
\newcommand{\bkappa}{\bar \kappa}

\def\lta{<}
\def\lea{\le}

\def\gea{\ge}

\def\lee{\preceq}

\newbox\noforkbox \newdimen\forklinewidth
\forklinewidth=0.3pt \setbox0\hbox{$\textstyle\smile$}
\setbox1\hbox to \wd0{\hfil\vrule width \forklinewidth depth-2pt
 height 10pt \hfil}
\wd1=0 cm \setbox\noforkbox\hbox{\lower 2pt\box1\lower
2pt\box0\relax}
\def\unionstick{\mathop{\copy\noforkbox}\limits}

\setbox0\hbox{$\textstyle\smile$}
\setbox1\hbox to \wd0{\hfil{\sl /\/}\hfil} \setbox2\hbox to
\wd0{\hfil\vrule height 10pt depth -2pt width
               \forklinewidth\hfil}
\newbox\doesforkbox
\setbox\doesforkbox\hbox{\lower 0pt\box1 \lower
2pt\box2\lower2pt\box0\relax}

\newcommand{\nf}{\unionstick}

\newcommand{\nfs}[4]{#2 \nf_{#1}^{#4} #3}

\def\1nf{\unionstick^{(1)}}
\def\2nf{\unionstick^{(2)}}

\newcommand{\tp}{\operatorname{tp}}
\newcommand{\qfl}{\qf\Ll_{\kappa, \kappa} (\bigtau)}
\newcommand{\qtp}{\tp_{\qfl}}
\newcommand{\gtp}{\text{gtp}}

\newcommand{\Ss}{\text{S}}

\newcommand{\gS}{\text{gS}}

\newcommand{\hanf}[1]{h (#1)}
\newcommand{\hanfs}[1]{h^\ast (#1)}

\newcommand{\LS}{\text{LS}}

\newcommand{\FV}[1]{\text{FV} (#1)}

\newcommand{\plus}[1]{{#1}_r}
\newcommand{\kappap}{\plus{\kappa}}

\newcommand{\Ksatpp}[2]{{#1}^{#2\text{-sat}}}
\newcommand{\Ksatp}[1]{\Ksatpp{K}{#1}}

\newcommand{\Kmhp}[1]{K^{#1\text{-mh}}}
\newcommand{\Kmh}[1]{\Kmhp{\lambda}}

\newcommand{\K}{K}
\newcommand{\Ll}{\mathbb{L}}
\newcommand{\qf}{\operatorname{qf-}}

\title{Infinitary stability theory}
\author{Sebastien Vasey}
\thanks{This material is based upon work done while the author was supported by the Swiss National Science Foundation under Grant No.\ 155136.}
\email{sebv@cmu.edu}
\urladdr{http://math.cmu.edu/\textasciitilde svasey/}
\address{Department of Mathematical Sciences, Carnegie Mellon University, Pittsburgh, Pennsylvania, USA}

\date{\today } 

\begin{document}

\begin{abstract} 
We introduce a new device in the study of abstract elementary classes (AECs): Galois Morleyization, which consists in expanding the models of the class with a relation for every Galois (orbital) type of length less than a fixed cardinal $\kappa$. We show:

\begin{thm}[The semantic-syntactic correspondence]
  An AEC $K$ is fully $(<\kappa)$-tame and type short if and only if Galois types are syntactic in the Galois Morleyization.
\end{thm}

This exhibits a correspondence between AECs and the syntactic framework of stability theory inside a model. We use the correspondence to make progress on the stability theory of tame and type short AECs. The main theorems are:

\begin{thm}\label{stab-spectrum-abstract}
  Let $K$ be a $\LS (\K)$-tame AEC with amalgamation. The following are equivalent:

  \begin{enumerate}
    \item $K$ is Galois stable in some $\lambda \ge \LS (\K)$.
    \item $K$ does not have the order property (defined in terms of Galois types).
    \item There exist cardinals $\mu$ and $\lambda_0$ with $\mu \le \lambda_0 < \beth_{(2^{\LS (\K)})^+}$ such that $K$ is Galois stable in any $\lambda \ge \lambda_0$ with $\lambda = \lambda^{<\mu}$.
  \end{enumerate}
\end{thm}

\begin{thm}\label{coheir-syn-ab}
  Let $K$ be a fully $(<\kappa)$-tame and type short AEC with amalgamation, $\kappa = \beth_{\kappa} > \LS (K)$. If $K$ is Galois stable, then the class of $\kappa$-Galois saturated models of $K$ admits an independence notion ($(<\kappa)$-coheir) which, except perhaps for extension, has the properties of forking in a first-order stable theory.
\end{thm}
\end{abstract}

\maketitle

\tableofcontents

\section{Introduction}

Abstract elementary classes (AECs) are sometimes described as a purely semantic framework for model theory. It has been shown, however, that AECs are closely connected with more syntactic objects. See for example Shelah's presentation theorem \cite[Lemma 1.8]{sh88}, or Kueker's \cite[Theorem 7.2]{kueker2008} showing that an AEC with Löwenheim-Skolem number $\lambda$ is closed under $\Ll_{\infty, \lambda^+}$-elementary equivalence.

Another framework for non-elementary model theory is stability theory inside a model (introduced in Rami Grossberg's 1981 master thesis and studied for example\footnote{The definition of a model being stable appears already in \cite[Definition I.2.2]{shelahfobook78} but (as Shelah notes in the introduction to \cite[Chapter I]{sh300-orig}) this concept was not pursued further there.} in \cite{grossberg91-indisc, grossberg91} or \cite[Chapter I]{sh300-orig}, see \cite[Chapter V.A]{shelahaecbook2} for a more recent version). There the methods are very syntactic but it is believed (see for example the remark on p.\ 116 of \cite{grossberg91-indisc}) that they can help the resolution of more semantic questions, such as Shelah's categoricity conjecture for $\Ll_{\omega_1, \omega}$.

In this paper, we establish a correspondence between these two frameworks. We show that results from stability theory inside a model directly translate to results about \emph{tame} abstract elementary classes. Recall that an AEC is $(<\kappa)$-tame if its Galois (i.e.\ orbital) types are determined by their restrictions to domains of size less than $\kappa$. Tameness as a property of AEC was first isolated (from an argument in \cite{sh394}) by Grossberg and VanDieren \cite{tamenessone} and used to prove an upward categoricity transfer \cite{tamenessthree,tamenesstwo}. Boney \cite{tamelc-jsl} showed that tameness follows from the existence of large cardinals. Combined with the categoricity transfers of Grossberg-VanDieren and Shelah \cite{sh394}, this showed assuming a large cardinal axiom that Shelah's eventual categoricity conjecture holds if the categoricity cardinal is a successor.

The basic idea of the translation is the observation (appearing for example in \cite[p.~15]{tamelc-jsl} or \cite[p.~206]{lieberman2011}) that in a $(<\kappa)$-tame abstract elementary class, Galois types over domains of size less than $\kappa$ play a role analogous to first-order formulas. We make this observation precise by expanding the language of such an AEC with a relation symbol for each Galois type over the empty set of a sequence of length than $\kappa$, and looking at $\Ll_{\kappa, \kappa}$-formulas in the expanded language. We call this expansion the \emph{Galois Morleyization}\footnote{We thank Rami Grossberg for suggesting the name.} of the AEC. Thinking of a type as the set of its small restrictions, we can then prove the \emph{semantic-syntactic correspondence} (Theorem \ref{separation}): Galois types in the AEC correspond to quantifier-free syntactic types in its Galois Morleyization.

The correspondence gives us a new method to prove results in tame abstract elementary classes:

\begin{enumerate}
  \item Prove a syntactic result in the Galois Morleyization of the AEC (e.g.\ using tools from stability theory inside a model).
  \item Translate to a semantic result in the AEC using the semantic-syntactic correspondence.
  \item Push the semantic result further using known (semantic) facts about AECs, maybe combined with more hypotheses on the AEC (e.g.\ amalgamation).
\end{enumerate}

As an application, we prove Theorem \ref{stab-spectrum-abstract} in the abstract (see Theorem \ref{stab-spectrum}), which gives the equivalence between no order property and stability in tame AECs and generalizes one direction of the stability spectrum theorem of homogeneous model theory (\cite[Theorem 4.4]{sh3}, see also \cite[Corollary 3.11]{grle-homog}). The syntactic part of the proof is not new (it is a straightforward generalization of Shelah's first-order proof \cite[Theorem 2.10]{shelahfobook}) and we are told by Rami Grossberg that proving such results was one of the reason tameness was introduced (in fact theorems with the same spirit appear in \cite{tamenessone}). However we believe it is challenging to give a transparent proof of the result using Galois types only. The reason is that the classical proof uses local types and it is not clear how to naturally define them semantically.

The method has other applications: Theorem \ref{coheir-syn} (formalizing Theorem \ref{coheir-syn-ab} from the abstract) shows that in stable fully tame and short AECs, the coheir independence relation has some of the properties of a well-behaved independence notion. This is used in \cite{indep-aec-v5} to build a global independence notion from superstability. In \cite{bv-sat-v3}, we also use syntactic methods to investigate chains of Galois-saturated models.

Precursors to this work include Makkai and Shelah's study of classes of models of an $\Ll_{\kappa, \omega}$ theory for $\kappa$ a strongly compact cardinal \cite{makkaishelah}: there they prove \cite[Proposition 2.10]{makkaishelah} that Galois and syntactic $\Sigma_1 (\Ll_{\kappa, \kappa})$-types are the same (so in particular those classes are $(<\kappa)$-tame). One can see the results of this paper as a generalization to tame AECs. Also, the construction of the Galois Morleyization when $\kappa = \aleph_0$ (so the language remains finitary) appears in \cite[Section 2.4]{group-config-kangas}. Moreover it has been pointed out to us\footnote{By Jonathan Kirby.} that a device similar to Galois Morleyization is used in \cite[Section 3]{rosicky81} to present any concrete category as a class of models of an infinitary theory. However the use of Galois Morleyization to translate results of stability theory inside a model to AECs is new.

This paper is organized as follows. In section \ref{prelim-sec}, we review some preliminaries. In section \ref{sec-foundations}, we introduce \emph{functorial expansions}\footnote{These were called ``abstract Morleyizations'' in an early version of this paper. We thank John Baldwin for suggesting the new name.} of AECs and the main example: Galois Morleyization. We then prove the semantic-syntactic correspondence. In section \ref{thy-indep}, we investigate various order properties and prove Theorem \ref{stab-spectrum-abstract}. In section \ref{sec-coheir}, we study the coheir independence relation. Several of these sections have \emph{global hypotheses} which hold until the end of the section: see Hypotheses \ref{ftypes-hyp}, \ref{morley-hyp}, and \ref{coheir-hyp}.

We end with a note on how AECs compare to some other non first-order framework like homogeneous model theory (see \cite{sh3}). There is an example (due to Marcus, see \cite{marcus-counterexample}) of an $\Ll_{\omega_1, \omega}$-axiomatizable class which is categorical in all uncountable cardinals but does not have an $\aleph_1$-sequentially-homogeneous model. For $n < \omega$, an example due to Hart and Shelah (see \cite{hs-example, bk-hs}) has amalgamation, no maximal models, and is categorical in all $\aleph_k$ with $k \le n$, but no higher. By \cite{tamenesstwo}, the example cannot be $\aleph_k$-tame for $k < n$. However if $\kappa$ is a strongly compact cardinal, the example will be fully $(<\kappa)$-tame and type short by the main result of \cite{tamelc-jsl}. The discussion on p.\ 74 of \cite{baldwinbook09} gives more non-homogeneous examples.

In general, classes from homogeneous model theory or quasiminimal pregeometry classes (see \cite{quasimin}) are special cases of AECs that are always fully $(<\aleph_0)$-tame and type short. In this paper we work with the much more general assumption of $(<\kappa)$-tameness and type shortness for a possibly uncountable $\kappa$.

This paper was written while working on a Ph.D.\ thesis under the direction of Rami Grossberg at Carnegie Mellon University and I would like to thank Professor Grossberg for his guidance and assistance in my research in general and in this work specifically. I thank Will Boney for thoroughly reading this paper and providing invaluable feedback. I also thank Alexei Kolesnikov for valuable discussions on the idea of thinking of Galois types as formulas. I thank John Baldwin, Jonathan Kirby, and a referee for valuable comments.

\section{Preliminaries}\label{prelim-sec}

We review some of the basics of abstract elementary classes and fix some notation. The reader is advised to skim through this section quickly and go back to it as needed. 

\subsection{Set theoretic terminology}

\begin{defin}\label{kappa-r-def}
  Let $\kappa$ be an infinite cardinal.
  \begin{enumerate}
  \item Let $\kappap$ be the least regular cardinal greater than or equal to $\kappa$. That is, $\kappap$ is $\kappa^+$ if $\kappa$ is singular and $\kappa$ if $\kappa$ is regular.
  \item Let $\kappa^-$ be $\kappa$ if $\kappa$ is limit or the unique $\kappa_0$ such that $\kappa_0^+ = \kappa$ if $\kappa$ is a successor.
  \end{enumerate}
\end{defin}

We will often use the following function:

\begin{defin}[Hanf function]\label{hanf-def}
  For $\lambda$ an infinite cardinal, define $\hanf{\lambda} := \beth_{(2^{\lambda})^+}$. Also define $\hanfs{\lambda} := \hanf{\lambda^-}$.
\end{defin}

Note that for $\lambda$ infinite, $\lambda = \beth_\lambda$ if and only if for all $\mu < \lambda$, $h (\mu) < \lambda$.

\subsection{Syntax}\label{syntax-subsec}

The notation of this paper is standard, but since we will work with infinitary objects and need to be precise, we review the basics. We will often work with the logic $\Ll_{\kappa, \kappa}$, see \cite{dickmann-book} for the definition and basic results.

\begin{defin}\label{infinitary-def}
  An \emph{infinitary vocabulary} is a vocabulary where we also allow relation and function symbols of infinite arity. For simplicity, we require the arity to be an ordinal. An infinitary vocabulary is \emph{$(<\kappa)$-ary} if all its symbols have arity strictly less than $\kappa$. A \emph{finitary vocabulary} is a $(<\aleph_0)$-ary vocabulary.
\end{defin}

For $\tau$ an infinitary vocabulary, $\phi$ an $\Ll_{\kappa, \kappa} (\tau)$-formula and $\bx$ a sequence of variables, we write $\phi = \phi (\bx)$ to emphasize that the free variables of $\phi$ appear among $\bx$ (recall that a $\Ll_{\kappa, \kappa}$-formula must have fewer than $\kappa$-many free variables, but not all elements of $\bx$ need to appear as free variables in $\phi$, so we allow $\ell(\bx) \ge \kappa$). We use a similar notation for sets of formulas. When $\ba$ is an element in some $\tau$-structure and $\phi (\bx, \by)$ is a formula, we often abuse notation and say that $\psi (\bx) = \phi (\bx, \ba)$ is a formula (again, we allow $\ell (\ba) \ge \kappa$). We say $\phi (\bx, \ba)$ is a \emph{formula over $A$} if $\ba \in \fct{<\infty}{A}$.

\begin{defin}
  For $\phi$ a formula over a set, let $\FV{\phi}$ denote an enumeration of the free variables of $\phi$ (according to some canonical ordering on all variables). That is, fixing such an ordering, $\FV{\phi}$ is the smallest sequence $\bx$ such that $\phi = \phi (\bx)$. Let $\ell (\phi) := \ell (\FV{\phi})$ (it is an ordinal, but by permutting the variables we can usually assume without loss of generality that it is a cardinal), and $\dom{\phi}$ be the smallest set $A$ such that $\phi$ is over $A$. Define similarly the meaning of $\FV{p}$, $\ell (p)$, and $\dom{p}$ on a set $p$ of formulas.
\end{defin}

\begin{defin}
For $\tau$ an infinitary vocabulary, $M$ a $\tau$-structure, $A \subseteq |M|$, $\bb \in \fct{<\infty}{|M|}$, and $\Delta$ a set of $\tau$-formulas (in some logic), let\footnote{Of course, we have in mind a canonical sequence of variables $\bx$ of order type $\ell (\bb)$ that should really be part of the notation but (as is customary) we always omit this detail.}:

$$
\tp_{\Delta} (\bb / A; M) := \{\phi (\bx; \ba) \mid \phi (\bx, \by) \in \Delta\text{, } \ba \in \fct{\ell (\by)}{A} \text{, and } M \models \phi[\bb, \ba]\}
$$

We will most often work with $\Delta = \qf\Ll_{\kappa, \kappa}$, the set of \emph{quantifier-free} $\Ll_{\kappa, \kappa}$-formulas.
\end{defin}

\begin{defin}\label{stone-def}
For $M$ a $\tau$-structure, $\Delta$ a set of $\tau$-formulas, $A \subseteq |M|$, $\alpha$ an ordinal or $\infty$, let 

$$
\Ss_{\Delta}^{<\alpha} (A; M) := \{\tp_{\Delta} (\bb / A; M) \mid \bb \in \fct{<\alpha}{|M|}\}
$$

Define similarly the variations for $\le \alpha$, $\alpha$, etc. We write $\Ss_{\Delta} (A; M)$ instead of $\Ss_{\Delta}^1 (A; M)$.
\end{defin}

\subsection{Abstract classes}

We review the definition of an abstract elementary class. Abstract elementary classes (AECs) were introduced by Shelah in \cite{sh88}. The reader unfamiliar with AECs can consult \cite{grossberg2002} for an introduction. 

We first review more general objects that we will sometimes use. Abstract classes are already defined in \cite{grossbergbook}, while $\mu$-abstract elementary classes are introduced in \cite{mu-aec-toappear-v2}. We will mostly use them to deal with functorial expansions and classes of saturated models of an AEC.

\begin{defin}
  An \emph{abstract class} (AC for short) is a pair $(K, \lea)$, where:

  \begin{enumerate}
    \item $K$ is a class of $\tau$-structure, for some fixed infinitary vocabulary $\tau$ (that we will denote by $\tau (K)$). We say $(K, \lea)$ is \emph{$(<\mu)$-ary} if $\tau$ is $(<\mu)$-ary.
    \item $\lea$ is a partial order (that is, a reflexive and transitive relation) on $K$.
    \item If $M \lea N$ are in $K$ and $f: N \cong N'$, then $f[M] \lea N'$ and both are in $K$.
    \item If $M \lea N$, then $M \subseteq N$.
  \end{enumerate}

\end{defin}
\begin{remark}
  We do not always strictly distinguish between $K$ and $(K, \lea)$.
\end{remark}
\begin{notation}
  For $K$ an abstract class, $M, N \in K$, we write $M \lta N$ when $M \lea N$ and $M \neq N$.
\end{notation}

\begin{defin}\label{r-increasing-def}
  Let $K$ be an abstract class. A sequence $\seq{M_i : i < \delta}$ of elements of $K$ is \emph{increasing} if for all $i < j < \delta$, $M_i \lea M_j$. \emph{Strictly increasing} means $M_i \lta M_j$ for $i < j$. $\seq{M_i : i < \delta}$ is \emph{continuous} if for all limit $i < \delta$, $M_i = \bigcup_{j < i} M_j$.
\end{defin}

\begin{notation}
  For $K$ an abstract class, we use notations such as $K_\lambda$, $K_{\ge \lambda}$, $K_{<\lambda}$ for the models in $K$ of size $\lambda$, $\ge \lambda$, $<\lambda$, respectively.
\end{notation}

\begin{defin}
Let $(I,\le)$ be a partially-ordered set.  

\begin{enumerate}
  \item We say that $I$ is  \emph{$\mu$-directed} provided for every $J\subseteq I$ if $|J|<\mu$ then there exists $r\in I$ such that $r\geq s$ for all $s\in J$ (thus $\aleph_0$-directed is the usual notion of directed set)
  \item Let $(K,\lea)$ be an abstract class. An indexed system $\seq{M_i : i \in I}$ of models in $K$ is \emph{$\mu$-directed} if $I$ is a $\mu$-directed set and $i < j$ implies $M_i \lea M_j$.
\end{enumerate}
\end{defin}
	
\begin{defin}\label{mu-aec-def}
  Let $\mu$ be a regular cardinal and let $(K, \lea)$ be a $(<\mu)$-ary abstract class.  We say that $(K, \lea)$ is a $\mu$-\emph{abstract elementary class} ($\mu$-AEC for short) if:
	
\begin{enumerate}
    \item Coherence: If $M_0, M_1, M_2 \in K$ satisfy $M_0 \lea M_2$, $M_1 \lea M_2$, and $M_0 \subseteq M_1$, then $M_0 \lea M_1$;

    \item Tarski-Vaught axioms: Suppose $\seq{M_i \in K : i \in I}$ is a $\mu$-directed system. Then:

        \begin{enumerate}

            \item $\bigcup_{i \in I} M_i \in K$ and, for all $j \in I$, we have $M_j \lea \bigcup_{i \in I} M_i$.

            \item If there is some $N \in K$ so that for all $i \in I$ we have $M_i \lea N$, then we also have $\bigcup_{i \in I} M_i \lea N$.

        \end{enumerate}

    \item Löwenheim-Skolem-Tarski axiom: There exists a cardinal $\lambda = \lambda^{<\mu} \ge |L(K)| + \mu$ such that for any $M \in K$ and $A \subseteq |M|$, there is some $M_0 \lea M$ such that $A \subseteq |M_0|$ and $\|M_0\| \le |A|^{<\mu} + \lambda$. We write $\LS (K)$ for the minimal such cardinal\footnote{Pedantically, $\LS (K)$ really depends on $\mu$ but $\mu$ will always be clear from context.}.
\end{enumerate}

When $\mu = \aleph_0$, we omit it and simply call $K$ an \emph{abstract elementary class} (AEC for short).
\end{defin}

In any abstract class, we can define a notion of embedding:

\begin{defin}
  Let $K$ be an abstract class. We say a function $f: M \rightarrow N$ is a \emph{$K$-embedding} if $M, N \in K$ and $f: M \cong f[M] \lea N$. For $A \subseteq |M|$, we write $f: M \xrightarrow[A]{} N$ to mean that $f$ fixes $A$ pointwise. Unless otherwise stated, when we write $f: M \rightarrow N$ we mean that $f$ is an embedding.
\end{defin}

Here are three key structural properties an abstract class can have:

\begin{defin}
  Let $K$ be an abstract class.
  \begin{enumerate}
    \item $K$ has \emph{amalgamation} if for any $M_0 \lea M_\ell$ in $K$, $\ell = 1,2$, there exists $N \in K$ and $f_\ell : M_\ell \xrightarrow[M_0]{} N$.
    \item $K$ has \emph{joint embedding} if for any $M_\ell$ in $K$, $\ell = 1,2$, there exists $N \in K$ and $f_\ell : M_\ell \rightarrow N$.
    \item $K$ has \emph{no maximal models} if for any $M \in K$ there exists $N \in K$ with $M \lta N$.
  \end{enumerate}
\end{defin}

\subsection{Galois types}

Let $K$ be an abstract class. There is a well-known a semantic notion of types for $K$, Galois types, that was first introduced by Shelah in \cite[Definition II.1.9]{sh300-orig}. While Galois types are usually only defined over models, here we allow them to be over any set. This is not harder and is often notationally convenient\footnote{For example, types over the empty sets are used here in the definition of the Galois Morleyization. They appear implicitly in the definition of the order property in \cite[Definition 4.3]{sh394} and explicitly in \cite[Notation 1.9]{tamenessone}. They are also used in \cite{finitary-aec}.}. Note however that Galois types over sets are in general not too well-behaved. For example, they can sometimes fail to have an extension (in the sense that if we have $N, N' \in \K$, $A \subseteq |N| \cap |N'|$ and $p$ a Galois type over $A$ realized in $N$, then we may not be able to extend $p$ to a type over $N'$) if their domain is not an amalgamation base.

\begin{defin}\label{gtp-def} \
  \begin{enumerate}
    \item Let $K^3$ be the set of triples of the form $(\bb, A, N)$, where $N \in K$, $A \subseteq |N|$, and $\bb$ is a sequence of elements from $N$. 
    \item For $(\bb_1, A_1, N_1), (\bb_2, A_2, N_2) \in K^3$, we say $(\bb_1, A_1, N_1)E_{\text{at}} (\bb_2, A_2, N_2)$ if $A := A_1 = A_2$, and there exists $f_\ell : N_\ell \xrightarrow[A]{} N$ such that $f_1 (\bb_1) = f_2 (\bb_2)$.
    \item Note that $E_{\text{at}}$ is a symmetric and reflexive relation on $K^3$. We let $E$ be the transitive closure of $E_{\text{at}}$.
    \item For $(\bb, A, N) \in K^3$, let $\gtp (\bb / A; N) := [(\bb, A, N)]_E$. We call such an equivalence class a \emph{Galois type}. We write $\gtp_K (\bb / A; N)$ when $K$ is not clear from context.
    \item For $p = \gtp (\bb / A; N)$ a Galois type, define\footnote{It is easy to check that this does not depend on the choice of representatives.} $\ell (p) := \ell (\bb)$ and $\dom{p} := A$.
  \end{enumerate}
\end{defin}

We can go on to define the restriction of a type (if $A_0 \subseteq \dom{p}$, $I \subseteq \ell (p)$, we will write $p^I \rest A_0$ when the realizing sequence is restricted to $I$ and the domain is restricted to $A_0$), the image of a type under an isomorphism, or what it means for a type to be realized. Just as in \cite[Observation II.1.11.4]{shelahaecbook}, we have:

\begin{fact}\label{ap-eat}
  If $K$ has amalgamation, then $E = E_{\text{at}}$.
\end{fact}

Note that the proof goes through, even though we only have amalgamation over models, not over all sets.

\begin{remark}
  To gain further insight into the difference between $E$ and $E_{\text{at}}$, consider the following situation. Let $K$ be an AEC that does \emph{not} have amalgamation and assume we are given $M \lea N$, $a_1, a_2 \in |M|$, and $A \subseteq |M|$. Suppose we know that $(a_1, A, M) E_{\text{at}} (a_2, A, M)$. Then because $(a_\ell, A, N) E_{\text{at}} (a_\ell, A, M)$ for $\ell = 1,2$, we have that $(a_1, A, N) E (a_2, A, N)$, but we may \emph{not} have that $(a_1, A, N) E_{\text{at}} (a_2, A, N)$.
\end{remark}

We also have the basic monotonicity and invariance properties \cite[Observation II.1.11]{shelahaecbook}, which follow directly from the definition:

\begin{prop}\label{galois-types-basic}
  Let $K$ be an abstract class. Let $N \in K$, $A \subseteq |N|$, and $\bb \in \fct{<\infty}{|N|}$.

  \begin{enumerate}
  \item\label{basic-1} Invariance: If $f: N \cong_A N'$, then $\gtp (\bb / A; N) = \gtp (f (\bb) / A; N')$.
  \item\label{basic-2} Monotonicity: If $N \lea N'$, then $\gtp (\bb / A; N) = \gtp (\bb / A; N')$.
  \end{enumerate}
\end{prop}

Monotonicity says that when $N \lea N'$, the set of Galois types (over a fixed set $A$) realized in $N'$ is at least as big as the set of Galois types over $A$ realized in $N$ (using the notation below, $\gS (A; N) \subseteq \gS (A; N')$). When $A = M$ for $M \lea N$ (or $A = \emptyset$), we can further define the class $\gS (A)$ of \emph{all} Galois types over $A$ in the natural way. Assuming the existence of a monster model $\sea$ containing $A$, this is the same as the usual definition: all types over $A$ realized in $\sea$.

\begin{defin} \
  \begin{enumerate}
    \item Let $N \in K$, $A \subseteq |N|$, and $\alpha$ be an ordinal. Define:

      $$
      \gS^\alpha (A; N) := \{\gtp (\bb / A; N) \mid \bb \in \fct{\alpha}{|N|}\}
      $$
    \item For $M \in K$ and $\alpha$ an ordinal, let:

      $$
      \gS^\alpha (M) := \{p \mid \exists N \in K : M \lea N \text { and } p \in \gS^{\alpha} (M ; N)\}
      $$

    \item For $\alpha$ an ordinal, let:

      $$
      \gS^\alpha (\emptyset) := \bigcup_{N \in K} \gS^{\alpha} (\emptyset; N)
      $$
  \end{enumerate}

  When $\alpha = 1$, we omit it. Similarly define $\gS^{<\alpha}$, where $\alpha$ is allowed to be $\infty$.
\end{defin}
\begin{remark}
  When $\alpha$ is an ordinal, $\gS^{\alpha} (M)$ and $\gS^{\alpha} (\emptyset)$ could a priori be proper classes. However in reasonable cases (e.g.\ when $K$ is a $\mu$-AEC) they are sets. For example when $K$ is a $\mu$-AEC, an upper bound for $|\gS^{\alpha} (M)|$ is $2^{\left(\|M\| + \alpha + \LS (\K)\right)^{<\mu}}$.
\end{remark}

Next, we recall the definition of tameness), a locality property of types. Tameness was introduced by Grossberg and VanDieren in \cite{tamenessone} and used to get an upward stability transfer (and an upward categoricity transfer in \cite{tamenesstwo}). Later on, Boney showed in \cite{tamelc-jsl} that it followed from large cardinals and also introduced a dual property he called \emph{type shortness}. 

\begin{defin}[Definitions 3.1 and 3.3 in \cite{tamelc-jsl}]\label{shortness-def}
  Let $K$ be an abstract class and let $\Gamma$ be a class (possibly proper) of Galois types in $K$. Let $\kappa$ be an infinite cardinal.

  \begin{enumerate}
    \item $K$ is \emph{$(<\kappa)$-tame for $\Gamma$} if for any $p \neq q$ in $\Gamma$, if $A := \dom{p} = \dom{q}$, then there exists $A_0 \subseteq A$ such that $|A_0| < \kappa$ and $p \rest A_0 \neq q \rest A_0$.
    \item $K$ is \emph{$(<\kappa)$-type short for $\Gamma$} if for any $p \neq q$ in $\Gamma$, if $\alpha := \ell (p) = \ell (q)$, then there exists $I \subseteq \alpha$ such that $|I| < \kappa$ and $p^I \neq q^I$.
    \item $\kappa$-tame means $(<\kappa^+)$-tame, similarly for type short.
    \item We usually just say ``short'' instead of ``type short''.
    \item Usually, $\Gamma$ will be a class of types over models only, and we often specify it in words. For example, \emph{$(<\kappa)$-short for types of length $\alpha$} means $(<\kappa)$-short for $\bigcup_{M \in K} \gS^\alpha (M)$. 
    \item We say $K$ is $(<\kappa)$-tame if it is $(<\kappa)$-tame for types of length one.
    \item We say $K$ is \emph{fully} $(<\kappa)$-tame if it is $(<\kappa)$-tame for $\bigcup_{M \in K} \gS^{<\infty} (M)$, similarly for short. 
  \end{enumerate}
\end{defin}

We review the natural notion of stability in this context. The definition here is slightly unusual compared to the rest of the litterature: we define what it means for a \emph{model} to be stable in a given cardinal, and get a local notion of stability that is equivalent (in AECs) to the usual notion if amalgamation holds, but behaves better if amalgamation fails. Note that we count the number of types over an arbitrary set, not (as is common in AECs) only over models. In case the abstract class has a Löwenheim-Skolem number and we work above it this is equivalent, as any type in $\gS^{<\alpha} (A; N)$ can be extended\footnote{Note that this does not use any amalgamation because we work inside the same model $N$.} to $\gS^{<\alpha} (B; N)$ when $A \subseteq B$, so $|\gS^{<\alpha} (A; N)| \le |\gS^{<\alpha} (B; N)|$.

\begin{defin}[Stability]\label{stab-def}
  Let $K$ be an abstract class. Let $\alpha$ be a cardinal, $\mu$ be a cardinal. A model $N \in K$ is $(<\alpha)$-\emph{stable in $\mu$} if for all $A \subseteq |N|$ of size $\le \mu$, $|\gS^{<\alpha} (A; N)| \le \mu$. Here and below, $\alpha$-stable means $(< (\alpha^+))$-stable. We say ``stable'' instead of ``1-stable''.

  $K$ is \emph{$(<\alpha)$-stable in $\mu$} if every $N \in K$ is $(<\alpha)$-stable in $\mu$. $K$ is \emph{$(<\alpha)$-stable} if it is $(<\alpha)$-stable in unboundedly many cardinals.

  Define similarly \emph{syntactically stable} for syntactic types (in this paper the quantifier-free $\Ll_{\kappa, \kappa}$-types, where $\kappa$ is clear from context).
\end{defin}

The next fact spells out the connection between stability for types of different lengths and tameness.

\begin{fact}\label{stab-facts} 
  Let $K$ be an AEC and let $\mu \ge \LS (K)$.
  \begin{enumerate}
    \item \label{stab-facts-tb} \cite[Theorem 3.1]{longtypes-toappear-v2}: If $K$ is stable in $\mu$, $K_\mu$ has amalgamation, and $\mu^\alpha = \mu$, then $K$ is $\alpha$-stable in $\mu$.
    \item \label{stab-facts-tame} \cite[Corollary 6.4]{tamenessone}\footnote{The result we want can easily be seen to follow from the proof there: see \cite[Theorem 12.10]{baldwinbook09}.}: If $K$ has amalgamation, is $\mu$-tame, and stable in $\mu$, then $K$ is stable in all $\lambda$ such that $\lambda^\mu = \lambda$.
    \item \label{stab-length-equiv} If $K$ has amalgamation, is $\mu$-tame, and is stable in $\mu$, then $K$ is $\alpha$-stable (in unboundedly many cardinals), for all cardinals $\alpha$.
  \end{enumerate}
\end{fact}
\begin{proof}[Proof of (\ref{stab-length-equiv})]
  Given cardinals $\lambda_0 \ge \LS (K)$ and $\alpha$, let $\lambda := \left(\lambda_0\right)^{\alpha + \mu}$. Combining the first two statements gives us that $K$ is $\alpha$-stable in $\lambda$.
\end{proof}

Finally, we review the natural definition of saturation using Galois types. Note that we again give the local definitions (but they are equivalent to the usual ones assuming amalgamation).

\begin{defin}
  Let $K$ be an abstract class, $M \in K$ and $\mu$ be an infinite cardinal. 
  \begin{enumerate}
    \item For $N \gea M$, $M$ is \emph{$\mu$-saturated\footnote{Pedantically, we should really say ``Galois-saturated'' to differentiate this from being syntactically saturated. In this paper, we will only discuss Galois saturation.} in $N$} if for any $A \subseteq |M|$ of size less than $\mu$, any $p \in \gS^{<\mu} (A; N)$ is realized in $M$. 
    \item $M$ is \emph{$\mu$-saturated} if it is $\mu$-saturated in $N$ for all $N \gea M$. When $\mu = \|M\|$, we omit it.
    \item We write $\Ksatp{\mu}$ for the class of $\mu$-saturated models of $K_{\ge \mu}$ (ordered by the ordering of $K$).
  \end{enumerate}
\end{defin}
\begin{remark}\label{sat-rmk} \
  \begin{enumerate}
    \item We defined saturation also when $\mu \le \LS (K)$. This is why we look at types over sets and not only over models. In an AEC, when $\mu > \LS (K)$, this is equivalent to the usual definition (see also the remark before Definition \ref{stab-def}).
    \item We could similarly define what it means for a \emph{set} to be saturated in a model (this is useful in \cite{bv-sat-v3}).
    \item It is easy to check that if $K$ is an AEC with amalgamation and $\mu > \LS (K)$, then $\Ksatp{\mu}$ is a $\plus{\mu}$-AEC (recall Definitions \ref{kappa-r-def} and \ref{mu-aec-def}) with $\LS (\Ksatp{\mu}) \le \LS (K)^{<\plus{\mu}}$.
  \end{enumerate}
\end{remark}

\section{The semantic-syntactic correspondence}\label{sec-foundations}

\subsection{Functorial expansions and the Galois Morleyization}

\begin{defin}\label{expansion-def}
  Let $K$ be an abstract class. A \emph{functorial expansion} of $K$ is a class $\bigK$ satisfying the following properties:

  \begin{enumerate}
    \item $\bigK$ is a class of $\bigtau$-structures, where $\bigtau$ is a fixed (possibly infinitary) vocabulary extending $\tau (\K)$.
    \item The map $\bigM \mapsto \bigM \rest \tau (\K)$ is a bijection from $\bigK$ onto $K$. For $M \in K$, we will write $\bigM$ for the unique element of $\bigK$ whose reduct is $M$. When we write ``$\bigM \in \bigK$'', it is understood that $M = \bigM \rest \tau (K)$.
    \item Invariance: For $M, N \in K$, if $f: M \cong N$, then $f: \bigM \cong \bigN$.
    \item Monotonicity: If $M \lea N$ are in $K$, then $\bigM \subseteq \bigN$.
  \end{enumerate}

  We say a functorial expansion $\bigK$ is \emph{$(<\kappa)$-ary} if $\tau(\bigK)$ is $(<\kappa)$-ary.
\end{defin}

\begin{example}\label{morleyization-example} \
  \begin{enumerate}
    \item For $K$ an abstract class, $K$ is a functorial expansion of $K$ itself. This is because $\lea$ must extend $\subseteq$.
    \item Let $K$ be an abstract class with $\tau := \tau (K)$ and let $\kappa$ be an infinite cardinal. Add a $(<\kappa)$-ary predicate $P$ to $\tau$, forming a language $\bigtau$. Expand each $M \in K$ to a $\bigL$-structure by defining $P^{\bigM} (\ba)$ (where $P^{\bigM}$ is the interpretation of $P$ inside $\bigM$) to hold if and only if $\ba$ is the universe of a $\lea$-submodel of $M$ (this is more or less what Shelah does in \cite[Definition IV.1.9.1]{shelahaecbook}). Then the resulting class $\bigK$ is a functorial expansion of $K$.
    \item Let $T$ be a complete first-order theory in a vocabulary $\tau$. Let $K := (\text{Mod} (T), \lee)$. It is common to expand $\tau$ to $\bigtau$ by adding a relation symbol for every first-order $\tau$-formula. We then expand $T$ (to $\widehat{T}$) and every model $M$ of $T$ in the expected way (to some $\bigM$) and obtain a new theory in which every formula is equivalent to an atomic one (this is commonly called the \emph{Morleyization} of the theory). Then $\bigK := \text{Mod} (\widehat{T})$ is a functorial expansion of $K$.
    \item Let $T$ be a first-order complete theory. Expanding each model $M$ of $T$ to its canonical model $M^{\text{eq}}$ of $T^{\text{eq}}$ (see \cite[III.6]{shelahfobook}) also describes a functorial expansion.
    \item The canonical structures of \cite{cherlin-harrington-lachlan} also induce a functorial expansion.
  \end{enumerate}
\end{example}

The main example of functorial expansion used in this paper is the \emph{Galois Morleyization}:

\begin{defin}\label{def-galois-m}
Let $K$ be an abstract class and let $\kappa$ be an infinite cardinal. Define an expansion $\bigtau$ of $\tau (K)$ by adding a relation symbol $R_p$ of arity $\ell (p)$ for each $p \in \gS^{<\kappa} (\emptyset)$. Expand each $N \in K$ to a $\bigtau$-structure $\bigN$ by specifying that for each $\ba \in \fct{<\kappa}{|\bigN|}$, $R_p^{\bigN} (\ba)$ (where $R_p^{\bigN}$ is the interpretation of $R_p$ inside $\bigN$) holds exactly when $\gtp (\ba / \emptyset; N) = p$. We call $\bigK$ the \emph{$(<\kappa)$-Galois Morleyization} of $K$.
\end{defin}
\begin{remark}\label{bigk-size}
  Let $K$ be an AEC and $\kappa$ be an infinite cardinal. Let $\bigK$ be the $(<\kappa)$-Galois Morleyization of $K$. Then $|\tau (\bigK)| \le |\gS^{<\kappa} (\emptyset)| + |\tau| \le 2^{<(\kappa + \LS (K)^+)}$.
\end{remark}

It is straightforward to check that the Galois Morleyization is a functorial expansion. We include a proof here for completeness.

\begin{prop}
  Let $K$ be an abstract class and let $\kappa$ be an infinite cardinal. Let $\bigK$ be the $(<\kappa)$-Galois Morleyization of $K$. Then $\bigK$ is a functorial expansion of $K$.
\end{prop}
\begin{proof}
    Let $\tau := \tau (\K)$ be the vocabulary of $\K$. Looking at Definition \ref{expansion-def}, there are four properties to check:

    \begin{enumerate}
    \item By definition of the Galois Morleyization, $\bigK$ is a class of $\bigtau$-structure, for a fixed vocabulary $\bigtau$.
    \item The map $\bigM \mapsto \bigM \rest \tau$ is a bijection: It is a surjection by definition of the Galois Morleyization. It is an injection: Assume that $M' := \bigM \rest \tau = \bigN \rest \tau$ but $\bigM \neq \bigN$. Then there must exist a $p \in \gS (\emptyset)$ and an $\ba \in \fct{<\kappa}{|M'|}$ such that $\gtp (\ba / \emptyset; M') = p$ but $\gtp (\ba / \emptyset; M') \neq p$. Thus $p \neq p$, a contradiction.
    \item Let $M, N \in \K$ and $f: M \cong N$. We have to see that $f: \bigM \cong \bigN$. Let $p \in \gS (\emptyset)$ and let $\ba \in \fct{<\kappa}{|M|}$. Assume that $\bigM \models R_p (\ba)$. Then by definition $p = \gtp (\ba / \emptyset; M)$. Therefore by Proposition \ref{galois-types-basic}.(\ref{basic-1}), $p = \gtp (f (\ba) / \emptyset; N)$. Hence $\bigN \models R_p (f (\ba))$. The steps can be reversed to obtain the converse.
    \item Let $M \lea N$ be in $\K$. We want to see that $\bigM \subseteq \bigN$. So let $p \in \gS (\emptyset)$, $\ba \in \fct{<\kappa}{|M|}$. Assume first that $\bigM \models R_p (\ba)$. Then $p = \gtp (\ba / \emptyset; M)$. Therefore by Proposition \ref{galois-types-basic}.(\ref{basic-2}), $p = \gtp (\ba / \emptyset; N)$. Therefore $\bigN \models R_p (\ba)$. The steps can be reversed to obtain the converse.
    \end{enumerate}  
\end{proof}

Note that a functorial expansion can naturally be seen as an abstract class:

\begin{defin}
  Let $(K, \lea)$ be an abstract class and let $\bigK$ be a functorial expansion of $K$. Define an ordering $\widehat{\lea}$ on $\bigK$ by $\bigM \widehat{\lea} \bigN$ if and only if $M \lea N$.
\end{defin}
\begin{remark}
  For simplicity, we will abuse notation and write $(\bigK, \lea)$ rather than $(\bigK, \widehat{\lea})$. As usual, when the ordering is clear from context we omit it.
\end{remark}

The next propositions are easy but conceptually interesting.

\begin{prop}\label{expansion-monot} 
  Let $(K, \lea)$ be an abstract class with $\tau := \tau (K)$. Let $\bigK$ be a functorial expansion of $K$ and let $\bigtau := \tau (\bigK)$.

  \begin{enumerate}
    \item $(\bigK, \lea)$ is an abstract class.
    \item If every chain in $K$ has an upper bound, then every chain in $\bigK$ has an upper bound.
    \item Galois types are the same in $K$ and $\bigK$: $\gtp_{K} (\ba_1 / A; N_1) = \gtp_{K} (\ba_2 / A; N_2)$ if and only if $\gtp_{\bigK} (\ba_1 / A; \bigN_1) = \gtp_{\bigK} (\ba_2 / A; \bigN_2)$.
    \item Assume $K$ is a $\mu$-AEC and $\bigK$ is a $(<\mu)$-ary Morleyization of $K$. Then $(\bigK, \lea)$ is a $\mu$-AEC with $\LS (\bigK) = \LS (K) + |\bigtau|^{<\mu}$.
    \item Let $\tau \subseteq \bigtau' \subseteq \bigtau$. Then $\bigK \rest \bigtau' := \{\bigM \rest \bigtau' \mid \bigM \in \bigK\}$ is a functorial expansion of $K$.
    \item If $\widehat{\bigK}$ is a functorial expansion\footnote{Where of course we think of $\bigK$ as an abstract class with the ordering induced from $K$.} of $\bigK$, then $\widehat{\bigK}$ is a functorial expansion of $K$.
  \end{enumerate}
\end{prop} 
\begin{proof}
  All are straightforward. As an example, we show that if $K$ is a $\mu$-AEC, $\bigK$ is a $(<\mu)$-ary functorial expansion of $K$, and $\seq{\bigM_i : i \in I}$ is a $\mu$-directed system in $\bigK$, then letting $M := \bigcup_{i \in I} M_i$, we have that $\bigcup_{i \in I} \bigM_i = \bigM$ (so in particular $\bigcup_{i \in I} \bigM_i \in \bigK$). Let $R$ be a relation symbol in $\bigtau$ of arity $\alpha$. Let $\ba \in \fct{\alpha}{|\bigM|}$. Assume $\bigM \models R[\ba]$. We show $\bigcup_{i \in I} \bigM_i \models R[\ba]$. The converse is done by replacing $R$ by $\neg R$, and the proof with function symbols is similar. Since $\bigtau$ is $(<\mu)$-ary, $\alpha < \mu$. Since $I$ is $\mu$-directed, $\ba \in \fct{\alpha}{|M_j|}$ for some $j \in I$. Since $M_j \lea M$, the monotonicity axiom implies $\bigM_j \subseteq \bigM$. Thus $\bigM_j \models R[\ba]$, and this holds for all $j' \ge j$. Thus by definition of the union, $\bigcup_{i \in I} \bigM_i \models R[\ba]$.
\end{proof}
\begin{remark}
  A word of warning: if $K$ is an AEC and $\bigK$ is a functorial expansion of $K$, then $K$ and $\bigK$ are isomorphic as categories. In particular, any directed system in $\bigK$ has a colimit. However, if $\tau (\bigK)$ is not finitary the colimit of a directed system in $\bigK$ may \emph{not} be the union: relations may need to contain more elements.
\end{remark}

\subsection{Formulas and syntactic types}

From now on until the end of the section, we assume:

\begin{hypothesis}\label{ftypes-hyp}
  $K$ is an abstract class with $\tau := \tau (K)$, $\kappa$ is an infinite cardinal, $\bigK$ is an arbitrary $(<\kappa)$-ary functorial expansion of $K$ with vocabulary $\bigtau := \tau (\bigK)$.
\end{hypothesis}

At the end of this section, we will specialize to the case when $\bigK$ is the $(<\kappa)$-Galois Morleyization of $K$. Recall from Section \ref{syntax-subsec} that the set $\qf\Ll_{\kappa, \kappa} (\bigtau)$ denotes the quantifier-free $L_{\kappa, \kappa} (\bigtau)$ formulas.

\begin{prop}\label{inv-lem}
  Let $\phi (\bx)$ be a quantifier-free $\Ll_{\kappa, \kappa} (\bigtau)$ formula, $M \in K$, and $\ba \in M$. If $f: M \rightarrow N$, then $\bigM \models \phi[\ba]$ if and only if $\bigN \models \phi[f (\ba)]$.
\end{prop}
\begin{proof}
  Directly from the invariance and monotonicity properties of functorial expansions.
\end{proof}

In general, Galois types (computed in $\K$) and syntactic types (computed in $\bigK$) are different. However, Galois types are always at least as fine as quantifier-free syntactic types (this is a direct consequence of Proposition \ref{inv-lem} but we include a proof for completeness).

\begin{lem}\label{gtp-map}
  Let $N_1, N_2 \in K$, $A \subseteq |N_\ell|$ for $\ell = 1,2$. Let $\bb_\ell \in N_\ell$. If $\gtp (\bb_1 / A; N_1) = \gtp (\bb_2 / A; N_2)$, then $\qtp (\bb_1 / A; \bigN_1) = \qtp (\bb_2 / A; \bigN_2)$.
\end{lem}
\begin{proof}
  By transitivity of equality, it is enough to show that if $(\bb_1, A, N_1) E_{\text{at}} (\bb_2, A, N_2)$, then $\qtp (\bb_1 / A; \bigN_1) = \qtp (\bb_2 / A; \bigN_2)$. So assume $(\bb_1, A, N_1) E_{\text{at}} (\bb_2, A, N_2)$. Then there exists $N \in K$ and $f_\ell : N_\ell \xrightarrow[A]{} N$ such that $f_1 (\bb_1) = f_2 (\bb_2)$. Let $\phi (\bx)$ be a quantifier-free $\Ll_{\kappa,\kappa} (\bigtau)$ formula over $A$. Assume $\bigN_1 \models \phi[\bb_1]$. By Proposition \ref{inv-lem}, $\bigN \models \phi[f_1 (\bb_1)]$, so $\bigN \models \phi[f_2 (\bb_2)]$, so by Proposition \ref{inv-lem} again, $\bigN_2 \models \phi[\bb_2]$. Replacing $\phi$ by $\neg \phi$, we get the converse, so $\qtp (\bb_1 / A; \bigN_1) = \qtp (\bb_2 / A; \bigN_2)$.
\end{proof}

Note that this used that the types were quantifier-free. We have justified the following definition:

\begin{defin}
  For a Galois type $p$, let $p^s$ be the corresponding quantifier-free syntactic type \emph{in the functorial expansion}. That is, if $p = \gtp (\bb / A; N)$, then $p^s := \qtp (\bb / A; \bigN)$.
\end{defin}





\begin{prop}\label{gtp-ftp-map}
  Let $N \in K$, $A \subseteq |N|$. Let $\alpha$ be an ordinal. The map $p \mapsto p^s$ from $\gS^\alpha (A; N)$ to $\Ss_{\qfl}^\alpha (A; \bigN)$ (recall Definition \ref{stone-def}) is a surjection.
\end{prop}
\begin{proof} 
  If $\qtp (\bb / A; \bigN) = q \in \Ss_{\qfl}^\alpha (A; \bigN)$, then by definition $\left(\gtp (\bb / A; N)\right)^s = q$.
\end{proof}

\begin{remark}
  To investigate formulas with quantifiers, we could define a different version of Galois types using isomorphisms rather than embeddings, and remove the monotonicity axiom from the definition of a functorial expansion. As we have no use for it here, we do not discuss this approach further.
\end{remark}

\subsection{On when Galois types are syntactic}

We have seen in Proposition \ref{gtp-ftp-map} that $p \mapsto p^s$ is a surjection, so Galois types are always at least as fine as quantifier-free syntactic type in the expansion. It is natural to ask when they are the same, i.e.\ when $p \mapsto p^s$ is a \emph{bijection}. When $\bigK$ is the $(<\kappa)$-Galois Morleyization of $K$ (see Definition \ref{def-galois-m}), we answer this using shortness and tameness (Definition \ref{shortness-def}). Note that we make no hypothesis on $K$. In particular, amalgamation is not needed.

\begin{thm}[The semantic-syntactic correspondence]\label{separation}
  Assume $\bigK$ is the $(<\kappa)$-Galois Morleyization of $\K$.

  Let $\Gamma$ be a family of Galois types. The following are equivalent:

  \begin{enumerate}
    \item\label{separation-1} $K$ is $(<\kappa)$-tame and short for $\Gamma$.
    \item\label{separation-2} The map $p \mapsto p^s$ is a bijection from $\Gamma$ onto $\Gamma^s := \{p^s \mid p \in \Gamma\}$.
  \end{enumerate}
\end{thm}
\begin{proof} \
  \begin{itemize}
    \item \underline{(\ref{separation-1}) implies (\ref{separation-2}):}
      By Lemma \ref{gtp-map}, the map $p \mapsto p^s$ with domain $\Gamma$ is well-defined and it is clearly a surjection onto $\Gamma^s$. It remains to see it is injective. Let $p, q \in \Gamma$ be distinct. If they do not have the same domain or the same length, then $p^s \neq q^s$, so assume that $A := \dom{p} = \dom{q}$ and $\alpha := \ell (p) = \ell (q)$. Say $p = \gtp (\bb / A; N)$, $q = \gtp (\bb' / A; N')$. By the tameness and shortness hypotheses, there exists $A_0 \subseteq A$ and $I \subseteq \alpha$ of size less than $\kappa$ such that $p_0 := p^I \rest A_0 \neq q^I \rest A_0 =: q_0$. Let $\ba_0$ be an enumeration of $A_0$, and let $\bb_0 := \bb \rest I$, $\bb_0' := \bb' \rest I$. Let $p_0' := \gtp (\bb_0 \ba_0 / \emptyset; N)$, and let $\phi := R_{p_0'} (\bx_0, \ba_0)$, where $\bx_0$ is a sequence of variables of type $I$. Since $\bb_0$ realizes $p_0$ in $N$, $\bigN \models \phi[\bb_0]$, and since $\bb_0'$ realizes $q_0$ in $N'$ and $q_0 \neq p_0$, $\widehat{N'} \models \neg \phi[\bb_0']$. Thus $\phi (\bx_0) \in p^s$, $\neg \phi (\bx_0) \in q^s$. By definition, $ \phi (\bx_0) \notin q$ so $p^s \neq q^s$.

    \item \underline{(\ref{separation-2}) implies (\ref{separation-1}):} Let $p, q \in \Gamma$ be distinct with domain $A$ and length $\alpha$. Say $p = \gtp (\bb / A; N)$, $q = \gtp (\bb' / A; N')$. By hypothesis, $p^s \neq q^s$ so there exists $\phi (\bx)$ over $A$ such that (without loss of generality) $\phi (\bx) \in p$ but $\neg \phi (\bx) \in q$. Let $A_0 := \dom{\phi}$, $\bx_0 := \FV{\phi}$ (note that $A_0$ and $\bx_0$ have size strictly less than $\kappa$). Let $\bb_0$, $\bb_0'$ be the corresponding subsequences of $\bb$ and $\bb'$ respectively. Let $p_0 := \gtp (\bb_0 / A_0; N)$, $q_0 := \gtp (\bb_0' / A_0; N')$. Then it is straightforward to check that $\phi \in p_0^s$, $\neg \phi \in q_0^s$, so $p_0^s \neq q_0^s$ and hence (by Lemma \ref{gtp-map}) $p_0 \neq q_0$. Thus $A_0$ and $I$ witness tameness and shortness respectively.
  \end{itemize}
\end{proof}
\begin{remark}
  The proof shows that (\ref{separation-2}) implies (\ref{separation-1}) is valid when $\bigK$ is any functorial expansion of $K$.
\end{remark}

\begin{cor}\label{types-formula} \
  Assume $\bigK$ is the $(<\kappa)$-Galois Morleyization of $\K$.

  \begin{enumerate}
    \item $K$ is fully $(<\kappa)$-tame and short if and only if for any $M \in K$ the map $p \mapsto p^s$ from $\gS^{<\infty} (M)$ to $\Ss_{\qfl}^{<\infty} (M)$ is a bijection\footnote{We have set $\Ss_{\qfl}^{<\infty} (M) := \bigcup_{N \gea M} \Ss_{\qfl}^{<\infty} (M; \bigN)$. Similarly define $\Ss_{\qfl} (M)$.}. 
    \item $K$ is $(<\kappa)$-tame if and only if for any $M \in K$ the map $p \mapsto p^s$ from $\gS (M)$ to $\Ss_{\qfl} (M)$ is a bijection. 
  \end{enumerate}
\end{cor}
\begin{proof}
  By Theorem \ref{separation} applied to $\Gamma := \bigcup_{M \in K} \gS^{<\infty} (M)$ and $\Gamma := \bigcup_{M \in K} \gS (M)$ respectively.
\end{proof}
\begin{remark}
  For $M \in K$, $p, q \in \gS (M)$, say $pE_{<\kappa}q$ if and only if $p \rest A_0 = q \rest A_0$ for all $A_0 \subseteq |M|$ with $|A_0| < \kappa$. Of course, if $K$ is $(<\kappa)$-tame, then $E_{<\kappa}$ is just equality. More generally, the proof of Theorem \ref{separation} shows that if $\bigK$ is the $(<\kappa)$-Galois Morleyization of $K$, then $p E_{<\kappa} q$ if and only if $p^s = q^s$. Thus in that case quantifier-free syntactic types in the Morleyization can be seen as $E_{<\kappa}$-equivalence classes of Galois types. Note that $E_{<\kappa}$ appears in the work of Shelah, see for example \cite[Definition 1.8]{sh394}.
\end{remark}

\section{Order properties and stability spectrum}\label{thy-indep}

In this section, we start applying the semantic-syntactic correspondence (Theorem \ref{separation}) to prove new structural results about AECs. In the introduction, we described a three-step general method to prove a result about AECs using syntactic methods. In the proof of Theorem \ref{stab-spectrum}, Corollary \ref{syn-stab-spectrum} gives the first step, Theorem \ref{separation} gives the second, while Facts \ref{shelah-hanf} (AECs have a Hanf number for the order property) and \ref{stab-facts} (In tame AECs with amalgamation, stability behaves reasonably well) are keys for the third step.

Throughout this section, we work with the $(<\kappa)$-Galois Morleyization of a fixed AEC $K$:

\begin{hypothesis}\label{morley-hyp} \
  \begin{enumerate}
  \item $K$ is an abstract elementary class.
  \item $\kappa$ is an infinite cardinal.
  \item $\bigK$ is the $(<\kappa)$-Galois Morleyization of $K$ (recall Definition \ref{def-galois-m}). Set $\bigtau := \tau (\bigK)$.
  \end{enumerate}
\end{hypothesis}

\subsection{Several order properties}

The next definition is a natural syntactic extension of the first-order order property. A related definition appears already in \cite{sh16} and has been well studied (see for example \cite{grsh222, grsh259}).

\begin{defin}[Syntactic order property]\label{syntactic-op}
  Let $\alpha$ and $\mu$ be cardinals with $\alpha < \kappa$. A model $\bigM \in \bigK$ has the \emph{syntactic $\alpha$-order property of length $\mu$} if there exists $\seq{\ba_i : i < \mu}$ inside $\bigM$ with $\ell (\ba_i) = \alpha$ for all $i < \mu$ and a \emph{quantifier-free} $\Ll_{\kappa, \kappa} (\bigtau)$-formula $\phi (\bx, \by)$ such that for all $i, j < \mu$, $\bigM \models \phi[\ba_i, \ba_j]$ if and only if $i < j$.

  Let $\beta \le \kappa$ be a cardinal. $\bigM$ has the \emph{syntactic $(<\beta)$-order property of length $\mu$} if it has the syntactic $\alpha$-order property of length $\mu$ for some $\alpha < \beta$. $\bigM$ has the \emph{syntactic order property of length $\mu$} if it has the syntactic $(<\kappa)$-order property of length $\mu$. 

  \emph{$\bigK$ has the syntactic $\alpha$-order of length $\mu$} if some $\bigM \in \bigK$ has it. \emph{$\bigK$ has the syntactic order property} if it has the syntactic order property for every length.
\end{defin}

We emphasize that the syntactic order property is always considered inside the Galois Morleyization $\bigK$ and must be witnessed by a \emph{quantifier-free} formula. Also, since any such formula has fewer than $\kappa$ free variables, nothing would be gained by defining the  $(\alpha)$-syntactic order property for $\alpha \ge \kappa$. Thus we talk of the syntactic order property instead of the $(<\kappa)$-syntactic order property.

Arguably the most natural semantic definition of the order property in AECs appears in \cite[Definition 4.3]{sh394}. For simplicity, we have removed one parameter from the definition.

\begin{defin}
  Let $\alpha$ and $\mu$ be cardinals. A model $M \in K$ has the \emph{Galois $\alpha$-order property of length $\mu$} if there exists $\seq{\ba_i : i < \mu}$ inside $M$ with $\ell (\ba_i) = \alpha$ for all $i < \mu$, such that for any $i_0 < j_0 < \mu$ and $i_1 < j_1 < \mu$, $\gtp (\ba_{i_0} \ba_{j_0} / \emptyset; N) \neq \gtp (\ba_{j_1} \ba_{i_1} / \emptyset; N)$.

  We usually drop the ``Galois'' and define variations such as ``$K$ has the $\alpha$-order property'' as in Definition \ref{syntactic-op}.
\end{defin}

Notice that the definition of the Galois $\alpha$-order property is more general than that of the syntactic $\alpha$-order property, since $\alpha$ is not required to be less than $\kappa$. However the next result shows that the two properties are equivalent when $\alpha < \kappa$. Notice that this does not use any tameness.

\begin{prop}\label{op-equiv}
  Let $\alpha$, $\mu$, and $\lambda$ be cardinals with $\alpha < \kappa$. Let $N \in K$.

  \begin{enumerate}
    \item If $\bigN$ has the syntactic $\alpha$-order property of length $\mu$, then $N$ has the $\alpha$-order property of length $\mu$.
    \item Conversely, let $\chi := |\gS^{\alpha + \alpha} (\emptyset)|$, and assume that $\mu \ge \left(2^{\lambda + \chi}\right)^+$. If $N$ has the $\alpha$-order property of length $\mu$, then $\bigN$ has the syntactic $\alpha$-order property of length $\lambda$.
  \end{enumerate}
  
  In particular, $K$ has the $\alpha$-order property if and only if $\bigK$ has the syntactic $\alpha$-order property.
\end{prop}

\begin{proof} \
  \begin{enumerate}
    \item This is a straightforward consequence of Proposition \ref{inv-lem}\footnote{We are using that everything in sight is quantifier-free. Note that this part works for any functorial expansion $\bigK$ of $K$.}.
    \item Let $\seq{\ba_i : i < \mu}$ witness that $N$ has the Galois $\alpha$-order property of length $\mu$. By the Erdős-Rado theorem used on the coloring $(i < j) \mapsto \gtp (\ba_i \ba_j / \emptyset; N)$, we get that (without loss of generality), $\seq{\ba_i : i < \lambda}$ is such that whenever $i < j$, $\gtp (\ba_i \ba_j / \emptyset; N) = p \in \gS^{\alpha + \alpha} (\emptyset)$. But then (since by assumption $\gtp (\ba_i \ba_j / \emptyset; N) \neq \gtp (\ba_j \ba_i / \emptyset; N)$), $\phi (\bx, \by) := R_p (\bx, \by)$ witnesses that $\bigN$ has the syntactic $\alpha$-order property of length $\lambda$.
  \end{enumerate}
\end{proof}

We will see later (Theorem \ref{stab-spectrum}) that assuming some tameness, \emph{even when $\alpha \ge \kappa$}, the $\alpha$-order property implies the syntactic order property.

In the next section, we heavily use the assumption of no syntactic order property \emph{of length $\kappa$}. We now look at how that assumption compares to the order property (of arbitrary long length). Note that Proposition \ref{op-equiv} already tells us that the $(<\kappa)$-order property implies the syntactic order property of length $\kappa$. To get an equivalence, we will assume $\kappa$ is a fixed point of the Beth function. The key is:

\begin{fact}\label{shelah-hanf}
  Let $\alpha$ be a cardinal. If $K$ has the $\alpha$-order property of length $\mu$ for all $\mu < \hanf{\alpha + \LS (K)}$, then $K$ has the $\alpha$-order property.
\end{fact}
\begin{proof}
By the same proof as \cite[Claim 4.5.3]{sh394}.
\end{proof}
\begin{cor}\label{hanf-limit-op} \
  Assume $\beth_\kappa = \kappa > \LS (K)$. Then $\bigK$ has the syntactic order property of length $\kappa$ if and only if $K$ has the $(<\kappa)$-order property.
\end{cor}
\begin{proof} 
  If $\bigK$ has the syntactic order property of length $\kappa$, then for some $\alpha < \kappa$, $\bigK$ has the syntactic $\alpha$-order property of length $\kappa$, and thus by Proposition \ref{op-equiv} the $\alpha$-order property of length $\kappa$. Since $\kappa = \beth_\kappa$, $\hanf{|\alpha| + \LS (K)} < \kappa$, so by Fact \ref{shelah-hanf}, $K$ has the $\alpha$-order property. 

  Conversely, if $K$ has the $(<\kappa)$-order property, Proposition \ref{op-equiv} implies that $\bigK$ has the syntactic order property, so in particular the syntactic order property of length $\kappa$.
\end{proof}

For completeness, we also discuss the following semantic variation of the syntactic order property of length $\kappa$ that appears in \cite[Definition 4.2]{bg-v9} (but is adapted from a previous definition of Shelah, see there for more background):

\begin{defin}\label{def-weak-op}
  For $\kappa > \LS (K)$, $K$ has the \emph{weak $\kappa$-order property} if there are $M \in K_{<\kappa}$, $N \gea M$, types $p \neq q \in \gS^{<\kappa} (M)$, and sequences $\seq{\ba_i : i < \kappa}$, $\seq{\bb_i : i < \kappa}$ from $N$ so that for all $i, j < \kappa$:

  \begin{enumerate}
  \item $i \le j$ implies $\gtp (\ba_i \bb_j / M; N) = p$.
  \item $i > j$ implies $\gtp (\ba_i \bb_j / M; N) = q$.
  \end{enumerate}
\end{defin}

\begin{lem}\label{weak-op-equiv}
  Let $\kappa > \LS (K)$.
  
  \begin{enumerate}
    \item If $K$ has the $(<\kappa)$-order property, then $K$ has the weak $\kappa$-order property.
    \item If $K$ has the weak $\kappa$-order property, then $\bigK$ has the syntactic order property of length $\kappa$.
  \end{enumerate}

  In particular, if $\kappa = \beth_\kappa$, then the weak $\kappa$-order property, the $(<\kappa)$-order property of length $\kappa$, and the $(<\kappa)$-order property are equivalent.
\end{lem}
\begin{proof} \
  \begin{enumerate}
    \item Assume $K$ has the $(<\kappa)$-order property. To see the weak order property, let $\alpha < \kappa$ be such that $K$ has the $\alpha$-order property. Fix an $N \in K$ such that $N$ has a long-enough $\alpha$-order property. Pick any $M \in K_{<\kappa}$ with $M \lea N$. By using the Erdős-Rado theorem twice, we can assume we are given $\seq{\bc_i : i < \kappa}$ such that whenever $i < j < \kappa$, $\gtp (\bc_i\bc_j / M; N) = p$, and $\gtp (\bc_j \bc_i / M; N) = q$, for some $p \neq q \in \gS (M)$.

  For $l < \kappa$, let $j_l := 2l$, and $k_l := 2l + 1$. Then $j_l, k_l < \kappa$, and $l \le l'$ implies $j_l < k_{l'}$, whereas $l > l'$ implies $j_l > k_{l'}$. Thus the sequences defined by $\ba_l := \bc_{j_l}$, $\bb_l := \bc_{k_l}$ are as required.
\item Assume $K$ has the weak $\kappa$-order property and let $M, N, p, q, \seq{\ba_i : i < \kappa}$, $\seq{\bb_i : i < \kappa}$ witness it. For $i < \kappa$, Let $\bc_i := \ba_i \bb_i$ and $\phi (\bx_1 \bx_2; \by_1 \by_2) := R_p (\by_1, \bx_2)$. This witnesses the syntactic order property of length $\kappa$ in $\bigK$.
  \end{enumerate}

  The last sentence follows from Proposition \ref{op-equiv} and Corollary \ref{hanf-limit-op}.
\end{proof}

\subsection{Order property and stability}

We now want to relate stability in terms of the number of types (see Definition \ref{stab-def}) to the order property and use this to find many stability cardinals.

Note that stability in $K$ (in terms of Galois types, see Definition \ref{stab-def}) coincides with syntactic stability in $\bigK$ given enough tameness and shortness (see Theorem \ref{separation}). In general, they could be different, but by Proposition \ref{inv-lem}, stability always implies syntactic stability (and so syntactic unstability implies unstability). This contrasts with the situation with the order properties, where the syntactic and regular order property are equivalent without tameness (see Proposition \ref{op-equiv}).

The basic relationship between the order property and stability is given by:

\begin{fact}\label{stab-facts-op}
  If $K$ has the $\alpha$-order property and $\mu \ge |\alpha| + \LS (K)$, then $K$ is not $\alpha$-stable in $\mu$. If in addition $\alpha < \kappa$, then $\bigK$ is not even syntactically $\alpha$-stable in $\mu$.
\end{fact}
\begin{proof}
  \cite[Claim 4.8.2]{sh394} is the first sentence. The proof (see \cite[Fact 5.13]{bgkv-v2}) generalizes (using the syntactic order property) to get the second sentence.
\end{proof}

This shows that the order property implies unstability and we now work towards a syntactic converse. The key is \cite[Theorem V.A.1.19]{shelahaecbook2}, which shows that if a model does not have the (syntactic) order property of a certain length, then it is (syntactically) stable in certain cardinals. Here, syntactic refers to Shelah's very general context, where any subset $\Delta$ of formulas from any abstract logic is allowed. Shelah assumes that the vocabulary is finitary but the proof goes through just as well with an infinitary vocabulary (the proof only deals with formulas, which \emph{are} allowed to be infinitary). Thus specializing the result to the context of this paper (working with the logic $\Ll_{\kappa, \kappa} (\bigtau)$ and $\Delta = \qfl$), we obtain:

\begin{fact}\label{syn-stab-spectrum-technical}
  Let $\bigN \in \bigK$. Let $\alpha < \kappa$. Let $\chi \ge (|\bigtau| + 2)^{<\kappa}$ be a cardinal. If $\bigN$ does not have the syntactic order property of length $\chi^+$, then whenever $\lambda = \lambda^{\chi} + \beth_2 (\chi)$, $\bigN$ is (syntactically) $(<\kappa)$-stable in $\lambda$.
\end{fact}

The next corollary does not need any amalgamation or tameness. Intuitively, this is because every property involved ends up being checked inside a single model (for example, $\bigK$ syntactically stable in some cardinal means that all of its \emph{models} are syntactically stable in the cardinal). 

\begin{cor}\label{syn-stab-spectrum}
  The following are equivalent:

  \begin{enumerate}
  \item\label{sss-some-card} For every $\kappa_0 < \kappa$ and every $\alpha < \kappa$, $\bigK$ is syntactically $\alpha$-stable in \emph{some} cardinal greater than or equal to $\LS (K) + \kappa_0$.
  \item\label{sss-op} $K$ does not have the $(<\kappa)$-order property.
  \item\label{sss-spec} There exist\footnote{The cardinal $\mu$ is closely related to the local character cardinal $\bkappa$ for nonsplitting. See for example \cite[Theorem 4.13]{tamenessone}.} cardinals $\mu$ and $\lambda_0$ with $\mu \le \lambda_0 < \hanfs{\kappa + \LS (K)^+}$ (recall Definition \ref{hanf-def}) such that $\bigK$ is syntactically $(<\kappa)$-stable in any $\lambda \ge \lambda_0$ with $\lambda^{<\mu} = \lambda$. In particular, $\bigK$ is syntactically $(<\kappa)$-stable.
  \end{enumerate}
\end{cor}
\begin{proof}
  (\ref{sss-spec}) says in particular that $\bigK$ is syntactically $(<\kappa)$-stable in a proper class of cardinals, so it clearly implies (\ref{sss-some-card}). (\ref{sss-some-card}) implies (\ref{sss-op}): We prove the contrapositive. Assume that $K$ has the $(<\kappa)$-order property. In particular, $K$ has the $(<\kappa)$-order property of length $\hanf{\kappa + \LS (\K)}$. By definition, this means that for some $\alpha < \kappa$, $K$ has the $\alpha$-order property of length $\hanf{\kappa + \LS (\K)}$. By Fact \ref{shelah-hanf}, $K$ has the $\alpha$-order property. By Fact \ref{stab-facts-op}, $\bigK$ is not syntactically $\alpha$-stable in any cardinal above $\LS (K) + |\alpha|$ (that is, for each $\lambda \ge \LS (\K) + |\alpha|$, there is $\bigN \in \bigK$ such that $\bigN$ is not syntactically $\alpha$-stable in $\lambda$). Thus taking $\kappa_0 := |\alpha|$, we get that (\ref{sss-some-card}) fails.

  Finally (\ref{sss-op}) implies (\ref{sss-spec}). Assume $K$ does not have the $(<\kappa)$-order property. By the contrapositive of Fact \ref{shelah-hanf}, for each $\alpha < \kappa$, there exists $\mu_\alpha < \hanf{|\alpha| + \LS (K)} \le \hanfs{\kappa + \LS (K)^+}$ such that $K$ does not have the  $\alpha$-order property of length $\mu_\alpha$. Since $2^{<(\kappa + \LS (K)^+)} < \hanfs{\kappa + \LS (K)^+}$, we can without loss of generality assume that $2^{<(\kappa + \LS (K)^+)} \le \mu_\alpha$ for all $\alpha < \kappa$. Let $\chi := \sup_{\alpha < \kappa} \mu_\alpha$. Then $K$ does not have the $(<\kappa)$-order property of length $\chi$. Now if $\kappa$ is a successor (say $\kappa = \kappa_0^+$), then $\chi = \mu_{\kappa_0} < \hanf{\kappa_0} \le \hanfs{\kappa + \LS (\K)^+}$. Otherwise $\hanfs{\kappa + \LS (\K)^+} = \hanf{\kappa + \LS (\K)}$ and $\cf{\hanf{\kappa + \LS (K)}} = (2^{\kappa + \LS (K)})^+ > \kappa$, so $\chi < \hanfs{\kappa + \LS (K)^+}$. Let $\mu := \chi^+$ and $\lambda_0 := \beth_2 (\chi)$. It is easy to check that $\mu \le \lambda_0 < \hanfs{\kappa + \LS (K)^+}$. Finally, note that by Remark \ref{bigk-size}, $|\bigtau| \le 2^{<(\kappa + \LS (K)^+)}$, so $\chi \ge (|\bigtau| + 2)^{<\kappa}$. Now apply Fact \ref{syn-stab-spectrum-technical} to each $\bigN \in \bigK$ (note that by definition of $\lambda_0$, if $\lambda = \lambda^{\chi} \ge \lambda_0$, then $\lambda = \lambda^{\chi} + \beth_2 (\chi)$).
\end{proof}
\begin{remark}
  Shelah \cite[Theorem 3.3]{sh932} claims (without proof) a version of (\ref{sss-some-card}) implies (\ref{sss-spec}).
\end{remark}

Assuming $(<\kappa)$-tameness for types of length less than $\kappa$, we can of course convert the above result to a statement about Galois types. To replace ``$(<\kappa)$-stable'' by just ``stable'' (recall that this means stable for types of length one) and also get away with only tameness for types of length one, we will use amalgamation together with Fact \ref{stab-facts}.

\begin{thm}\label{stab-spectrum}
  Assume $K$ has amalgamation and is $(<\kappa)$-tame. The following are equivalent:

  \begin{enumerate}
  \item\label{ss-some-card} $K$ is stable in \emph{some} cardinal greater than or equal to $\LS (K) + \kappa^-$ (recall Definition \ref{kappa-r-def}).
  \item\label{ss-op} $K$ does not have the order property.
  \item\label{ss-op2} $K$ does not have the $(<\kappa)$-order property.
  \item\label{ss-spec} There exist cardinals $\mu$ and $\lambda_0$ with $\mu \le \lambda_0 < \hanfs{\kappa + \LS (K)^+}$ (recall Definition \ref{hanf-def}) such that $K$ is stable in any $\lambda \ge \lambda_0$ with $\lambda^{<\mu} = \lambda$.
  \end{enumerate}

  In particular, $K$ is stable if and only if $K$ does not have the order property.
\end{thm}
\begin{proof} 
  Clearly, (\ref{ss-spec}) implies (\ref{ss-some-card}) and (\ref{ss-op}) implies (\ref{ss-op2}). (\ref{ss-some-card}) implies (\ref{ss-op}): If $K$ has the $\alpha$-order property, then by Fact \ref{stab-facts-op} it cannot be $\alpha$-stable in any cardinal above $\LS (K) + |\alpha|$. By Fact \ref{stab-facts}.(\ref{stab-length-equiv}), $K$ is not stable in any cardinal greater than or equal to $\kappa^- + \LS (K)$, so (\ref{ss-some-card}) fails. Finally, (\ref{ss-op2}) implies (\ref{ss-spec}) by combining Corollary \ref{syn-stab-spectrum} and Corollary \ref{types-formula}.
\end{proof}
\begin{proof}[Proof of Theorem \ref{stab-spectrum-abstract}]
  Set $\kappa := \LS (\K)^+$ in Theorem \ref{stab-spectrum}. Note that in that case $\kappa^- = \LS (\K)$ (Definition \ref{kappa-r-def}) and $\hanfs{\kappa + \LS (\K)^+} = \hanfs{\LS (\K)^+} = \hanf{\LS (\K)}$ by Definition \ref{hanf-def}.
\end{proof}

\section{Coheir}\label{sec-coheir}

We look at the natural generalization of coheir (introduced in \cite{lascar-poizat} for first-order logic) to the context of this paper. A definition of coheir for classes of models of an $\Ll_{\kappa, \omega}$ theory was first introduced in \cite{makkaishelah} and later adapted to general AECs in \cite{bg-v9}. We give a slightly more conceptual definition here and show that coheir has several of the basic properties of forking in a stable first-order theory. This improves on \cite{bg-v9} which assumed that coheir had the extension property.

\begin{hypothesis}\label{coheir-hyp} \
  \begin{enumerate}
    \item $K^0$ is an AEC with amalgamation.
    \item $\kappa > \LS (\K^0)$ is a fixed cardinal.
    \item $K := \Ksatpp{\left(K^0\right)}{\kappa}$ is the class of $\kappa$-saturated models of $K^0$.
    \item $\bigK$ is the $(<\kappa)$-Galois Morleyization of $K$. Set $\bigtau := \tau (\bigK)$.
  \end{enumerate}  
\end{hypothesis}

The reader can see $\bigK$ as the class in which coheir is computed syntactically, while $K$ is the class in which it is used semantically. 

For the sake of generality, we do not assume stability or tameness yet. We will do so in parts (\ref{coheir-2}) and (\ref{coheir-3}) of Theorem \ref{coheir-syn}, the main theorem of this section. After the proof of Theorem \ref{coheir-syn}, we give a proof of Theorem \ref{coheir-syn-ab} in the abstract.

Note that by Remark \ref{sat-rmk}, $K$ is a $\kappap$-AEC (see Definition \ref{mu-aec-def}). Moreover by saturation the ordering has some elementarity. More precisely, let $\Sigma_1 (\Ll_{\kappa, \kappa} (\bigtau))$ denote the set of $\Ll_{\kappa, \kappa} (\bigtau)$-formulas of the form $\exists \bx \psi (\bx; \by)$, where $\psi$ is quantifier-free. We then have:

\begin{prop}\label{elem-prop} 
  If $M, N \in K$ and $M \lea N$, then $\bigM \lee_{\Sigma_1 (\Ll_{\kappa, \kappa} (\bigtau))} \bigN$.
\end{prop}
\begin{proof}
Assume that $\bigN \models \exists \bx \psi (\bx; \ba)$, where $\ba \in \fct{<\kappa}{|M|}$ and $\psi$ is a quantifier-free $\Ll_{\kappa, \kappa} (\bigtau)$-formula. Let $A$ be the range of $\ba$. Let $\bb \in \fct{<\kappa}{|N|}$ be such that $\bigN \models \psi[\bb, \ba]$. Since $M$ is $\kappa$-saturated, there exists $\bb' \in \fct{<\kappa}{|M|}$ such that $\gtp (\bb' / A; M) = \gtp (\bb / A; N)$. Now it is easy to check using Proposition \ref{gtp-map} that $\bigM \models \psi[\bb'; \ba]$.
\end{proof}

Also note that if $\kappa$ is suitably chosen and $\K^0$ is stable, then we have a strong failure of the order property in $\bigK$:

\begin{prop}\label{syntactic-op-stable}
  If $\kappa = \beth_\kappa$ and $\K^0$ is stable (in unboundedly many cardinals, see Definition \ref{stab-def}), then $\bigK$ does not have the syntactic order property of length $\kappa$.
\end{prop}
\begin{proof}
  By Fact \ref{stab-facts}, $\K^0$ is $(<\kappa)$-stable in unboundedly many cardinals. By Fact \ref{stab-facts-op}, $K^0$ does not have the $(<\kappa)$-order property. 

  Let $\widehat{\K^0}$ be the $(<\kappa)$-Galois Morleyization of $\K^0$. By Corollary \ref{hanf-limit-op}, $\widehat{K^0}$ does not have the syntactic order property of length $\kappa$. 

  Now note that Galois types are the same in $\K$ and $\K^0$: for $N \in \K$, $A \subseteq |N|$, and $\bb, \bb' \in \fct{<\infty}{|N|}$, $\gtp_{K^0} (\bb / A; N) = \gtp_{K^0} (\bb' / A; N)$ if and only if $\gtp_{\K} (\bb / A; N) = \gtp_{\K} (\bb' / A; N)$\footnote{Recall that $\gtp_K$ denotes Galois types as computed in $K$ and $\gtp_{K^0}$ Galois types computed in $K^0$ (see Definition \ref{gtp-def}).}. To see this, use amalgamation together with the fact that every model in $K^0$ can be $\lea$-extended to a model in $K$. 

  It follows that $\bigK \subseteq \widehat{\K^0}$. By definition of the syntactic order property, this means that also $\bigK$ does not have the syntactic order property of length $\kappa$, as desired.
\end{proof}

\begin{defin}\label{heir-coheir-def}
  Let $\bigN \in \bigK$, $A \subseteq |\bigN|$, and $p$ be a set of formulas (in some logic) over $\bigN$.

  \begin{enumerate}
  \item $p$ is a \emph{$(<\kappa)$-heir over $A$} if for any formula $\phi (\bx; \bb) \in p$ over $A$, there exists $\ba \in \fct{<\kappa}{A}$ such that $\phi (\bx; \ba) \in p \rest A$.
  \item $p$ is a \emph{$(<\kappa)$-coheir over $A$ in $\bigN$} if for any $\phi (\bx) \in p$ there exists $\ba \in \fct{<\kappa}{A}$ such that $\bigN \models \phi[\ba]$. When $\bigN$ is clear from context, we drop it.
  \end{enumerate}
\end{defin}
\begin{remark}
  Here, $\kappa$ is fixed (Hypothesis \ref{coheir-hyp}), so we will just remove it from the notation and simply say that $p$ is a (co)heir over $A$.
\end{remark}
\begin{remark}
  In this section, $p$ will be $\qtp (\bc / B; \bigN)$ for a fixed $B$ such that $A \subseteq B \subseteq |\bigN|$.
\end{remark}
\begin{remark}
  Working in $\bigN \in \bigK$, let $\bc$ be a permutation of $\bc'$, and $A, B$ be sets. Then $\qtp (\bc / B; \bigN)$ is a coheir over $A$ if and only if $\qtp (\bc' / B; \bigN)$ is a coheir over $A$. Similarly for heir. Thus we can talk about $\qtp (C / B; \bigN)$ being a heir/coheir over $A$ without worrying about the enumeration of $C$.
\end{remark}

We will mostly look at coheir, but the next proposition tells us how to express one in terms of the other.

\begin{prop}
$\qtp (\ba / A \bb; \bigN)$ is a heir over $A$ if and only if $\qtp (\bb / A \ba; \bigN)$ is a coheir over $A$.
\end{prop}
\begin{proof}
  Straightforward.
\end{proof}

It is convenient to see coheir as an independence relation:

\begin{notation}
  Write $\nfs{M}{A}{B}{N}$ if $M, N \in K$, $M \lea N$, and $\qtp (A / |M| \cup B; \bigN)$ is a coheir over $|M|$ in $\bigN$. We also say\footnote{It is easy to check this does not depend on the choice of representatives.} that $\gtp (A / B; N)$ \emph{is a coheir over $M$}.
\end{notation}
\begin{remark}
  The definition of $\nf$ depends on $\kappa$ but we hide this detail. 
\end{remark}

Interestingly, Definition \ref{heir-coheir-def} is equivalent to the semantic definition of Boney and Grossberg \cite[Definition 3.2]{bg-v9}:

\begin{prop}
  Let $N \in K$.
  Then $p \in \gS^{<\infty} (B; N)$ is a coheir over $M \lea N$ if and only if for any $I \subseteq \ell (p)$ and any $B_0 \subseteq B$, if $|I_0| + |B_0| < \kappa$, $p^I \rest B_0$ is realized in $M$.
\end{prop}
\begin{proof}
  Straightforward
\end{proof}

For completeness, we show that the definition of heir also agrees with the semantic definition of Boney and Grossberg \cite[Definition 6.1]{bg-v9}.

\begin{prop}
  Let $M_0 \lea M \lea N$ be in $K$, $\ba \in \fct{<\infty}{|\bigN|}$.

  Then $\qtp (\ba / M; \bigN)$ is a heir over $M_0$ if and only if for all $(<\kappa)$-sized $I \subseteq \ell (\ba)$ and $(<\kappa)$-sized $M_0^- \lea M_0$, $M_0^- \lea M^- \lea M$ (where we also allow $M_0^-$ to be empty), there is $f: M^- \xrightarrow[M_0^-]{} M_0$ such that $\gtp (\ba / M; N)$ extends $f (\gtp ((\ba \rest I) / M^-; N))$.
\end{prop}
\begin{proof}
  Assume first $\qtp (\ba / M; \bigN)$ is a heir over $M_0$ and let $I \subseteq \ell (\ba)$, $M_0^- \lea M^- \lea M$ be $(<\kappa)$-sized, with $M_0^-$ possibly empty. Let $p := \gtp ((\ba \rest I) / M^-; N)$. Let $\bb_0$ be an enumeration of $M_0^-$ and let $\bb$ be an enumeration of $|M^-| \backslash |M_0^-|$. Let $q := \gtp ((\ba \rest I) \bb_0 \bb / \emptyset; N)$. Consider the formula $\phi (\bx; \bb; \bb_0) := R_q (\bx; \bb; \bb_0)$, where $\bx$ are the free variables in $\qtp (\ba / M; \bigN)$ and we assume for notational simplicity that the $I$-indiced variables are picked out by $R_q (\bx, \bb, \bb_0)$. Then $\phi$ is in $\qtp (\ba / M; \bigN)$. By the syntactic definition of heir, there is $\bc \in \fct{<\kappa}{|M_0|}$ such that $\phi (\bx; \bc; \bb_0)$ is in $\qtp (\ba / M_0; \bigN)$. By definition of the $(<\kappa)$-Galois Morleyization this means that $\gtp ((\ba \rest I) \bb \bb_0 / \emptyset; N) = \gtp ((\ba \rest I) \bc \bb_0 / \emptyset; N)$. 

  By definition of Galois types and amalgamation (see Fact \ref{ap-eat}), there exists $N' \gea N$ and $g: N \rightarrow N'$ such that $g ((\ba \rest I) \bb \bb_0) = (\ba \rest I) \bc \bb_0$. Let $f := g \rest M^-$. Then from the definitions of $\bb_0$, $\bb$, and $\bc$, we have that $f: M^- \xrightarrow[M_0^-]{} M_0$. Moreover, $f (\gtp ((\ba \rest I) / M^-; N)) = \gtp (\ba \rest I / f[M^-]; N)$, which is clearly extended by $\gtp (\ba / M; N)$.

  The converse is similar.
\end{proof}
\begin{remark}
  The notational difficulties encountered in the above proof and the complexity of the semantic definition of heir show the convenience of using a syntactic notation rather than working purely semantically.
\end{remark}

We now investigate the properties of coheir. For the convenience of the reader, we explicitly prove the uniqueness property (we have to slightly adapt the proof of (U) from \cite[Proposition 4.8]{makkaishelah}). For the others, they are either straightforward or we can just quote.

\begin{lem}\label{uniq-sym}
  Let $M, N, N' \in K$ with $M \lea N$, $M \lea N'$. Assume $\bigM$ does \emph{not} have the syntactic order property of length $\kappa$. Let $\ba \in \fct{<\infty}{|N|}$, $\ba' \in \fct{<\infty}{|N'|}$, $\bb \in \fct{<\infty}{|M|}$ be given such that:
  \begin{enumerate}
    \item $\qtp (\ba / M; \bigN) = \qtp (\ba' / M; \widehat{N'})$
    \item $\qtp (\ba / M \bb; \bigN)$ is a coheir over $M$.
    \item $\qtp (\bb / M \ba'; \widehat{N'})$ is a coheir over $M$.
  \end{enumerate}

  Then $\qtp (\ba / M \bb; \bigN) = \qtp (\ba' / M \bb; \widehat{N'})$.
\end{lem}
\begin{proof}
  We suppose not and prove that $\bigM$ has the syntactic order property of length $\kappa$. Assume that $\qtp (\ba / M \bb; \bigN) \neq \qtp (\ba' / M\bb; \widehat{N'})$ and pick $\phi (\bx, \by)$ a formula over $M$ witnessing it: 

\begin{equation}
\bigN \models \phi[\ba; \bb] \text{ but } \widehat{N'} \models \neg \phi[\ba'; \bb] \label{eq2}
\end{equation}

(note that we can assume without loss of generality that $\ell (\ba) + \ell (\bb) < \kappa$).

Define by induction on $i < \kappa$ $\ba_i$, $\bb_i$ \emph{in $M$} such that for all $i, j < \kappa$:

  \begin{enumerate}
    \item $\bigM \models \phi[\ba_i, \bb]$.
    \item $\bigM \models \phi[\ba_i, \bb_j]$ if and only if $i \le j$.
    \item\label{cond3} $\bigN \models \neg \phi[\ba, \bb_i]$.
  \end{enumerate}

  Note that since $\bb_i \in \fct{<\kappa}{|M|}$, (\ref{cond3}) is equivalent to $\widehat{N'} \models \neg \phi[\ba', \bb_i]$.

  \underline{This is enough:} Then $\chi (\bx_1, \by_1, \bx_2, \by_2) := \phi (\bx_1, \by_2) \land \bx_1 \by_1 \neq \bx_2 \by_2$ together with the sequence $\seq{\ba_i\bb_i : i < \kappa}$  witness the syntactic order property of length $\kappa$.

  \underline{This is possible:} Suppose that $\ba_j, \bb_j$ have been defined for all $j < i$. Note that by the induction hypothesis and \eqref{eq2} we have:

  $$
  \bigN \models \bigwedge_{j < i} \phi [\ba_j, \bb] \land \bigwedge_{j < i} \neg \phi [\ba, \bb_j] \land \phi [\ba, \bb]
  $$

  Since $\qtp (\ba / A \bb; \bigN)$ is a coheir over $M$, there is $\ba'' \in \fct{<\kappa}{|M|}$ such that: 

  $$
  \bigN \models \bigwedge_{j < i} \phi [\ba_j, \bb] \land \bigwedge_{j < i} \neg \phi [\ba'', \bb_j] \land \phi [\ba'', \bb]
  $$

  Note that all the data in the equation above is in $M$, so as $M \lea N$, the monotonicity axiom of functorial expansions implies $\bigM \subseteq \bigN$, so $\bigM$ also models the above. By monotonicity again, $\widehat{N'}$ models the above. We also know that $\widehat{N'} \models \neg \phi[\ba', \bb]$. Thus we have:

  $$
  \widehat{N'} \models \bigwedge_{j < i} \phi [\ba_j, \bb] \land \bigwedge_{j < i} \neg \phi [\ba'', \bb_j] \land \phi [\ba'', \bb] \land \neg \phi[\ba', \bb]
  $$

 Since $\qtp (\bb / M \ba'; \bigN)$ is a coheir over $M$, there is $\bb'' \in \fct{<\kappa}{|M|}$ such that:   

  $$
 \widehat{N'} \models \bigwedge_{j < i} \phi [\ba_j, \bb''] \land \bigwedge_{j < i} \neg \phi [\ba'', \bb_j] \land \phi [\ba'', \bb''] \land \neg \phi [\ba', \bb'']
  $$

 Let $\ba_i := \ba''$, $\bb_i := \bb''$. It is easy to check that this works.  
\end{proof}

\begin{thm}[Properties of coheir]\label{coheir-syn} \
  \begin{enumerate}
    \item\label{coheir-1} \begin{enumerate}
      \item Invariance: If $f: N \cong N'$ and $\nfs{M}{A}{B}{N}$, then $\nfs{f[M]}{f[A]}{f[B]}{N'}$. 
      \item Monotonicity: If $\nfs{M}{A}{B}{N}$ and $M \lea M' \lea N_0 \lea N$, $A_0 \subseteq A$, $B_0 \subseteq B$, $|M'| \subseteq B$, $A_0 \cup B_0 \subseteq |N_0|$, then $\nfs{M'}{A_0}{B_0}{N_0}$.
      \item Normality: If $\nfs{M}{A}{B}{N}$, then $\nfs{M}{A \cup |M|}{B \cup |M|}{N}$.
      \item Disjointness: If $\nfs{M}{A}{B}{N}$, then $A \cap B \subseteq |M|$.
      \item Left and right existence: $\nfs{M}{A}{M}{N}$ and $\nfs{M}{M}{B}{N}$.
      \item Left and right $(<\kappa)$-set-witness: $\nfs{M}{A}{B}{N}$ if and only if for all $A_0 \subseteq A$ and $B_0 \subseteq B$ of size less than $\kappa$, $\nfs{M}{A_0}{B_0}{N}$.
      \item Strong left transitivity: If $\nfs{M_0}{M_1}{B}{N}$ and $\nfs{M_1}{A}{B}{N}$, then $\nfs{M_0}{A}{B}{N}$.
    \end{enumerate}
    \item\label{coheir-2} If $\bigK$ does not have the syntactic order property of length $\kappa$, then\footnote{Note that (by Proposition \ref{syntactic-op-stable}) this holds in particular if $\kappa = \beth_\kappa$ and $K^0$ is stable.}:
      \begin{enumerate}
        \item Symmetry: $\nfs{M}{A}{B}{N}$ if and only if $\nfs{M}{B}{A}{N}$.
        \item Strong right transitivity: If $\nfs{M_0}{A}{M_1}{N}$ and $\nfs{M_1}{A}{B}{N}$, then $\nfs{M_0}{A}{B}{N}$.
        \item Set local character: For all cardinals $\alpha$, all $p \in \gS^{\alpha} (M)$, there exists $M_0 \lea M$ with $\|M_0\| \le \mu_\alpha := \left(\alpha + 2\right)^{<\kappap}$ such that $p$ is a coheir over $M_0$.
        \item\label{coheir-syn-uq} Syntactic uniqueness: If $M_0 \lea M \lea N_\ell$ for $\ell = 1,2$, $|M_0| \subseteq B \subseteq |M|$. $q_\ell \in \Ss_{\qfl}^{<\infty} (B; \widehat{N_\ell})$, $q_1 \rest M_0 = q_2 \rest M_0$ and $q_\ell$ is a coheir over $M_0$ in $\widehat{N_\ell}$ for $\ell = 1,2$, then $q_1 = q_2$.

        \item\label{stab-noop} Syntactic stability: For $\alpha$ a cardinal, $\bigK$ is syntactically $\alpha$-stable in all $\lambda \ge \LS (K^0)$ such that $\lambda^{\mu_\alpha} = \lambda$.
          \end{enumerate}
    \item\label{coheir-3} If $\bigK$ does not have the syntactic order property of length $\kappa$ and $K^0$ is $(<\kappa)$-tame and short for types of length less than $\alpha$, then:
      \begin{enumerate}
        \item Uniqueness: If $p, q \in \gS^{<\alpha} (M)$ are coheir over $M_0 \lea M$ and $p \rest M_0 = q \rest M_0$, then $p = q$. 
        \item Stability: For all $\beta < \alpha$, $K^0$ is $\beta$-stable in all $\lambda \ge \LS (\K^0)$ such that $\lambda^{\mu_\beta} = \lambda$.
      \end{enumerate}
    \end{enumerate}
\end{thm}
\begin{proof} 
  Observe that (except for part (\ref{coheir-3})), one can work in $\bigK$ and prove the properties there using purely syntactic methods (so amalgamation is never needed for example). More specifically, (\ref{coheir-1}) is straightforward. As for (\ref{coheir-2}), symmetry is exactly as in\footnote{Note that a proof of symmetry of nonforking from no order property already appears in \cite{shelahfobook78}, but Pillay's proof for coheir is the one we use here.} \cite[Proposition 3.1]{pillay82} (Lemma \ref{uniq-sym} is not needed here), strong right transitivity follows from strong left transitivity and symmetry, syntactic uniqueness is by symmetry and Lemma \ref{uniq-sym}, and set local character is as in the proof of $(B)_\mu$ in \cite[Proposition 4.8]{makkaishelah}. Note that the proofs in \cite{makkaishelah} and \cite{pillay82} use that the ordering has some elementarity. In our case, this is given by Proposition \ref{elem-prop}.

  The proof of stability is as in the first-order case. To get part (\ref{coheir-3}), use the translation between Galois and syntactic types (Theorem \ref{separation}).
\end{proof}
\begin{proof}[Proof of Theorem \ref{coheir-syn-ab}]
  If the hypotheses of Theorem \ref{coheir-syn-ab} in the abstract hold for the AEC $K^0$, then the hypothesis of each parts of Theorem \ref{coheir-syn} hold (see Proposition \ref{syntactic-op-stable}).
\end{proof}
\begin{remark}
  We can give more localized version of some of the above results. For example in the statement of the symmetry property it is enough to assume that $\bigM$ does not have the syntactic order property of length $\kappa$. We could also have been more precise and state the uniqueness property in terms of being $(<\kappa)$-tame and short for $\{q_1, q_2\}$, where $q_1, q_2$ are the two Galois types we are comparing.
\end{remark}
\begin{remark}
  We can use Theorem \ref{coheir-syn}.(\ref{stab-noop}) to get another proof of the equivalence between (syntactic) stability and no order property in AECs.
\end{remark}
\begin{remark}
  The extension property (given $p \in \gS^{<\infty} (M)$, $N \gea M$, $p$ has an extension to $N$ which is a coheir over $M$) seems more problematic. In \cite{bg-v9}, Boney and Grossberg simply assumed it (they also showed that it followed from $\kappa$ being strongly compact \cite[Theorem 8.2.1]{bg-v9}). Here we do not need to assume it but are still unable to prove it. In \cite{indep-aec-v5}, we prove it assuming a superstability-like hypothesis and more locality\footnote{A word of caution: In \cite[Section 4]{hyttinen-lessmann}, the authors give an example of an $\omega$-stable class that does not have extension. However, the extension property they consider is \emph{over all sets}, not only over models.}.
\end{remark}

\bibliographystyle{amsalpha}
\bibliography{syntax-tame-aecs}

\end{document}